\documentclass[letterpaper, 10 pt, conference]{ieeeconf}  

\IEEEoverridecommandlockouts                             

\usepackage{graphics} 
\usepackage{epsfig} 
\usepackage{mathptmx} 
\usepackage{times} 
\usepackage{amsmath} 
\usepackage{amssymb}  
\usepackage{xcolor}

\title{\LARGE \bf
A probabilistic max-plus numerical method for 
solving stochastic control problems}

\author{Marianne Akian$^{1}$ and Eric Fodjo$^{2}$
\thanks{$^{1}$Marianne Akian is with 
INRIA and CMAP, \'Ecole polytechnique CNRS. Address:
CMAP, \'Ecole Polytechnique,
Route de Saclay,
91128 Palaiseau Cedex, France.
        {\tt\small marianne.akian@inria.fr}}%
\thanks{$^{2}$Eric Fodjo is with
I-Fihn Consulting and 
INRIA and CMAP, \'Ecole polytechnique CNRS. Address:
CMAP, \'Ecole Polytechnique,
Route de Saclay,
91128 Palaiseau Cedex, France.
        {\tt\small eric.fodjo@polytechnique.edu}}%
}

\newcommand{\figperso}[2]{\begin{figure}[htb]\begin{center} #1\end{center}
#2\end{figure}}

\newcommand{\R}{\mathbb R}

\newcommand{\Sy}{\mathbb S}
\newcommand{\E}{\mathbb E}
\newcommand{\M}{\mathcal M}
\newcommand{\T}{\mathcal T}
\newcommand{\U}{\mathcal U}
\newcommand{\F}{\mathcal F}
\newcommand{\Q}{\mathcal Q}
\newcommand{\G}{\mathcal G}
\newcommand{\W}{\mathcal W}

\newcommand{\D}{\mathcal D}
\newcommand{\Po}{\mathcal P}
\newcommand{\Li}{\mathcal L}
\newcommand{\Ha}{\mathcal H}
\newcommand{\tr}{\mathop{\mathrm{tr}}}
\newcommand{\tp}{^{\mathsf T}\,}
\newcommand{\card}[1]{\# #1}

\newcommand{\wT}{\widetilde{T}}
\newcommand{\wG}{\widetilde{G}}
\newcommand{\A}{\mathfrak A}
\newcommand{\NXW}{N_{\mathrm{rg}}}
\newcommand{\NX}{N_x}
\newcommand{\NW}{N_w}
\newcommand{\Nzero}{N_{\mathrm{in}}}
\newcommand{\Nm}{N_{\mathrm{m}}}
\newcommand{\eps}{\epsilon}

\newtheorem{theorem}{Theorem}
\newtheorem{lemma}[theorem]{Lemma}
\newtheorem{algorithm}{Algorithm}

\begin{document}

\maketitle
\thispagestyle{empty}
\pagestyle{empty}

\begin{abstract}

We consider fully nonlinear Hamilton-Jacobi-Bellman equations
associated to diffusion control problems involving 
a finite set-valued (or switching) control and possibly 
a continuum-valued control.
We construct a  lower complexity probabilistic numerical algorithm
by combining the idempotent expansion properties obtained by 
McEneaney, Kaise and Han (2011) for solving such problems
with a numerical probabilistic method such as the one 
proposed by Fahim, Touzi and Warin (2011)
for solving some fully nonlinear parabolic partial differential
equations.
Numerical tests on a small example of pricing and hedging an option
are presented.
\end{abstract}

\section{Introduction}\label{sec-int}

We consider a finite horizon diffusion control problem on $\R^d$ involving 
at the same time a ``discrete'' control taking its values in a finite set $\M$, 
and a ``continuum'' control taking its values in some subset $\U$ 
of a finite dimensional space $\R^p$ (for instance a 
convex set with nonempty interior), which we next describe.

Let $T$ be the horizon. 
The state $\xi_s\in\R^d$ at time $s\in [0,T]$
satisfies the stochastic differential equation
\begin{equation}\label{defxi}
d \xi_s = f^{\mu_s} (\xi_s, u_s) ds + \sigma^{\mu_s} (\xi_s, u_s) d W_s
\enspace , \end{equation}
where $(W_s)_{s\geq 0}$ is a $d$-dimensional
Brownian motion on a filtered probability space
$(\Omega,\F,(\F_s)_{0\leq s\leq T},P)$.
The control processes $\mu:=(\mu_s)_{0\leq s\leq T}$ and $u:=(u_s)_{0\leq s\leq T}$
take their values in the
sets $\M$ and $\U$ respectively and they are admissible if
they are  progressively measurable with respect to
the filtration $(\F_s)_{0\leq s\leq T}$.
We assume that, for all $m\in\M$,
the maps $f^m: \R^d\times \U\to \R^d$ and
$\sigma^m: \R^d\times \U\to \R^{d\times d}$ are continuous
and satisfy properties implying the existence of the process 
$(\xi_s)_{0\leq s\leq T}$ for any admissible control processes $\mu$ and $u$.

Given an initial time $t\in [0,T]$, the control problem consists in maximizing
the following payoff:
\begin{align*}
&J(t, x, \mu, u) := \E \left[ \int_t^T 
e^{ - \int_t^s \delta^{\mu_{\tau}} (\xi_{\tau}, u_{\tau}) d{\tau} }
\ell^{\mu_s} (\xi_s, u_s) ds \right.\\
&\left. \quad + e^{ - \int_t^T \delta^{\mu_{\tau}} (\xi_{\tau}, u_{\tau}) d{\tau} }
\psi(\xi_T)  \mid \xi_t = x \right] \enspace ,
\end{align*}
where, for all $m\in \M$, $\ell^m: \R^d\times \U\to \R$,
$\delta^m: \R^d\times \U\to \R_+$ (the set of positive reals),
and $\psi:\R^d\to\R$ are  given continuous maps.
We then define the value function of the problem as the optimal payoff:
$$v(t, x) = \sup_{\mu, u} J(t,x,\mu, u)\enspace ,$$
where the maximization holds over all admissible control processes
$\mu$ and $u$.

Let $\Sy_d$ denotes the set of symmetric $d\times d$ matrices.
The Hamiltonian $\Ha:\R^d\times \R \times \R^d\times \Sy_d\to \R$
of the above control problem is defined as:
\[ \Ha(x, r,p,\Gamma):=\max_{m\in \M} \Ha^m(x, r,p,\Gamma)\enspace,\]
with 
\begin{align*}
 \Ha^m(x, r,p,\Gamma):=&\max_{u\in \U} \Big\{
\frac{1}{2} \tr\left(\sigma^m(x,u) \sigma^m(x,u)\tp \Gamma
\right)  \\
& +f^m(x,u)\cdot p-\delta^m(x,u)r +\ell^m(x,u)\Big\}\enspace.
\end{align*}
Under suitable assumptions, the value function 
$v:[0,T]\times \R^d\to \R$ is the unique 
(continuous) viscosity solution of
the following Hamilton-Jacobi-Bellman equation
\begin{align}
&-\frac{\partial v}{\partial t} 
-\Ha(x, v(t,x), Dv(t,x), D^2v(t,x))=0, \label{HJB}\\
&\; \hspace{10em} x\in \R^d,\; t\in [0,T), \nonumber\\
& v(T,x)=\psi(x), \quad x\in \R^d,\nonumber
\end{align}
satisfying also some growth condition at infinity (in space).

In~\cite{touzi2011}, Fahim, Touzi and Warin proposed a probabilistic numerical 
method to solve such fully nonlinear partial 
differential equations~\eqref{HJB}, inspired by their backward stochastic
differential equation interpretation 
given by Cheridito, Soner, Touzi and Victoir
in~\cite{cheridito2007}.
However this method only works when 
the diffusion matrices $\sigma^m(x,u) \sigma^m(x,u)\tp$ are 
at the same time bounded from below (with respect to the Loewner order)
by a symmetrix positive definite matrix $a$ and 
bounded from above by $(1+2/d) a$.
Such a constraint can be restrictive, in particular it may not hold
even when the matrices  $\sigma^m(x,u)$ do not depend on $x$ and $u$ but take
different values for $m\in \M$. Also some regularity conditions
may be needed for $\Ha$, which are not fulfilled 
when $\M$ is a finite set.

McEneaney, Kaise and Han
proposed in~\cite{mceneaney2010,mceneaney2011}
an idempotent numerical method which works at least when 
the hamiltonian $\Ha^m$ corresponds to linear quadratic control problems.
This method is based on the distributivity of the (usual) addition operation
over the supremum (or infimum) operation, and on a property
of invariance of the set of quadratic forms.
It computes in a backward manner the value function $v(t,\cdot)$
at time $t$ as a supremum of quadratic forms.
However, as $t$ decreases, the number of quadratic forms generated 
by the method increases exponentially (and even become infinite
if the Brownian is not discretized in space) and some pruning 
is necessary to reduce the complexity of the algorithm.

Here, we combine the  two above methods to construct a
new algorithm.
The method is using in particular the simulation of a small number of
uncontrolled stochastic processes as in~\cite{touzi2011}.
We show that
even without pruning, the complexity of the algorithm
is bounded polynomially in the number of discretization time steps
and in the size of the sample of the uncontrolled stochastic processes.
Numerical tests of our algorithm 
 on an example of pricing and hedging an option in dimension 2
considered in~\cite{kharroubi2014} are presented.

\section{The algorithm of  Fahim, Touzi and Warin}\label{sec-touzi}

Let $h$ be a time discretization step such that $T/h$ is an integer.
We denote by $\T_h=\{0,h,2h,\ldots, T-h\}$ the set of discretization times
of $[0,T)$.

For each $m\in\M$, we shall assume that we can apply 
the algorithm of~\cite{touzi2011} to the equation:
\begin{align}
&-\frac{\partial v}{\partial t} 
-\Ha^m(x, v(t,x), Dv(t,x), D^2v(t,x))=0, \label{m-eq}\\
&\; \hspace{10em} x\in \R^d,\; t\in [0,T).\nonumber
\end{align}
For this purpose, we decompose $\Ha^m$ as the sum of the (linear) generator 
$\Li^m$ of a diffusion (with no control) and of
a nonlinear elliptic Hamiltonian $\G^m$, that is
$\Ha^m=\Li^m+\G^m$ with 
\begin{align*}
 \Li^m(x, r,p,\Gamma):=&
\frac{1}{2} \tr\left(a^m(x)  \Gamma \right) +\underline{f}^m(x)\cdot p \enspace,
\end{align*}
$a^m(x)=\underline{\sigma}^m(x) \underline{\sigma}^m(x)\tp$ 
and $\G^m$ such that 
$\partial_{\Gamma} {\G^m}$ is positive semidefinite, for all $x\in\R^d,\;
r\in\R, p\in \R^d, \Gamma\in \Sy_d$.
We also assume that 
$\tr(a^m(x)^{-1} \partial_{\Gamma} {\G^m})\leq 1$.

The time discretization of \eqref{m-eq} proposed in~\cite{touzi2011} 
can be written in the following form:
\[ v^h(t,x)=T_{t,h}^m(v^h(t+h,\cdot))(x),\quad t\in\T_h\enspace,\]
where, under some conditions, $T_{t,h}^m$ is a monotone operator
over the set of Lipschitz continuous functions from $\R^d$ to $\R$:
\begin{equation}\label{monotone}
 \phi\leq \psi\implies T_{t,h}^m(\phi)\leq T_{t,h}^m(\psi)\enspace .\end{equation}
Moreover,  the operator $T_{t,h}^m$ is constructed by using a
probabilistic scheme.
Denote by $\hat{X}^{m}$
the Euler discretization of the diffusion with generator $\Li^m$:
\begin{eqnarray}
\lefteqn{ \hat{X}^{m}(t+h)=}\nonumber \\
&& \label{xhat} 
 \hat{X}^{m}(t)+ \underline{f}^m(\hat{X}^{m}(t))h+
\underline{\sigma}^m(\hat{X}^{m}(t)) (W_{t+h} -W_t)\enspace .\end{eqnarray}
Then, 
\begin{eqnarray}
\lefteqn{  T_{t,h}^m(\phi)(x)= \D_{m,t,h}^{0}(\phi)(x)}\nonumber\\
&&+ h \G^m(x, \D_{m,t,h}^{0}(\phi)(x),\D_{m,t,h}^{1}(\phi)(x),\D_{m,t,h}^{2}(\phi)(x))
\label{def-th-m}
\end{eqnarray}
with, for $i=0,1,2$, 
$ \D_{m,t,h}^i(\phi)$ being the approximation of the $i$th derivative of $\phi$
obtained as follows:
\[ \D_{m,t,h}^i(\phi)(x)=
\E(\phi(\hat{X}^{m}(t+h))
\Po^i_{m,t,x,h}(W_{t+h}-W_t)\mid \hat{X}^{m}(t)=x)
\enspace ,\]
where, for all $m,t,x,h,i$, $\Po^i_{m,t,x,h}$ is
a polynomial of degree $i$
 with values in an appropriate finite dimensional space,
and in particular $\Po^0_{m,t,x,h}\equiv 1$.
Although the operator $T_{t,h}^m$ 
does not depend on $t$, since both the law of $W_{t+h}-W_t$ 
and the Hamiltonian $\Ha^m$ do not depend on $t$,
we keep the index $t$ since it will become
important when applying a regression approximation (see below).

In~\cite{touzi2011}, the convergence of such a time discretization is
proved under the above assumptions and some other technical assumptions.
Note that these conditions include the boundedness of the  coefficients 
of the Hamiltonians  $\Li^m$ and $\Ha^m$, and the boundedness of 
the value function of the corresponding control problem.
However, a change of variable on the value and on the state 
allows one to obtain the same type of result for
unbounded coefficients and value function satisfying some suitable
growth conditions at infinity.
Such a change of variables may induce a deformation on the 
discretizations~\eqref{xhat} and~\eqref{def-th-m} and so on the algorithm, but we shall
not discuss this here.
A greater difficulty is that the  above assumptions do not allow 
in general to handle directly
the case where $\Ha^m$ is replaced by $\Ha$.
However, in the case of $\Ha$ as above one can simply consider the
following scheme:
\begin{equation}
 v^h(t,x)=\max_{m\in\M} \{T_{t,h}^m(v^h(t+h,\cdot))(x)\},\quad t\in\T_h\enspace.
\label{disc1}
\end{equation}
The difference with the usual scheme of~\cite{touzi2011} is that
one needs to construct several operators $T_{t,h}^m$ and so several
processes  $\hat{X}^{m}$, one for each $m\in\M$.
The solution $v^h$ of this time discretization will converge to the
value function of our problem, that is the solution
of the Hamilton-Jacobi-Bellman equation with hamiltonian $\Ha$,
 as soon as the convergence 
is proved for the time discretization of the equations
with hamiltonians $\Ha^m$.

Although the above scheme can be compared to a standard numerical
approximation if one develops the expression
of each $T_{t,h}^m(\phi)(x)$, with $m\in\M$, 
one may compute $v^h$ given by~\eqref{disc1} as in~\cite{touzi2011},
that is using a regression estimator.
One just simulates the process $\hat{X}^{m}$ and
do at each time $t\in\T_h$ a regression estimation to find the value
of $\D_{m,t,h}^i(v^h(t+h,\cdot))$ at the points $\hat{X}^{m}(t)$ by using
the values of  $\hat{X}^{m}(t+h)$ and $W_{t+h}-W_t$.

Although this variation of the method of~\cite{touzi2011}, based 
on~\eqref{disc1},
is appealing and may work in practice,
several difficulties remain.
First, theoretically, the sample size to obtain
the convergence of the estimator  is at least
in the order of $1/h^{d/2}$~\cite{bouch-tou}.
Hence, it is exponential in the dimension of the system
showing the persistence of  the curse of dimensionality,
although in some practical examples, a much smaller sample size
may be sufficient.
Next, one possible regression estimation is to 
approximate the conditional expectation of a random map 
by projecting it orthogonally into a finite dimensional
linear space of functions. Then, to obtain a good estimation,
the dimension of this space need to be exponential in the dimension $d$.
In the sequel, we shall rather use a
small dimensional regression space and use a distributivity property as in
the work of  McEneaney, Kaise and Han~\cite{mceneaney2010,mceneaney2011}
to find a good approximation of $v_h$ living in the max-plus linear space
of finite suprema of quadratic forms.

\section{The algorithm of  McEneaney, Kaise and Han}\label{sec-bill}
In~\cite{mceneaney2010,mceneaney2011},  the following time
discretization is used.
Denote by $\hat{\xi}^{m,u}$
the Euler discretization of the process $\xi$ defined in the introduction,
when the controls $m$ and $u$ are fixed:
\begin{align*}
\hat{\xi}^{m,u}(t+h)=&\hat{\xi}^{m,u}(t)+ f^m(\hat{\xi}^{m,u}(t),u)h+\\
& \qquad\sigma^m(\hat{\xi}^{m,u}(t),u) (W_{t+h} -W_t)\enspace .\end{align*}
Then, a time discretization of the solution of~\eqref{HJB}
is given by:
\begin{equation}\label{disc2}
 v^h(t,x)=T_{t,h}(v^h(t+h,\cdot))(x),\quad t\in\T_h\enspace,\end{equation}
where
\begin{align}
 T_{t,h}(\phi)(x)=&\sup_{m\in\M, u\in \U} 
\Big\{ h \ell^{m} (x,u)
+e^{ - h \delta^{m} (x,u)} \nonumber \\
& \qquad  \qquad \E \big[ \phi(\hat{\xi}^{m,u}(t+h))
 \mid \hat{\xi}^{m,u}(t)=x \big] \Big\}\enspace ,
\label{defdiscop2}\end{align}
Under appropriate assumptions,
this scheme converges to the solution of~\eqref{HJB}
(see~\cite{mceneaney2011} for $\delta^m=0$).
Note that the processes $\hat{X}^{m}$  of the 
previous section are not related to the processes $\hat{\xi}^{m,u}$
and so the above discretization is different from the one
of the previous section.

Assume that the final reward $\psi$ of the control problem
can be written as the supremum of a finite number 
of concave quadratic forms.
Denote $\Q_d=\Sy_d^-\times \R^d\times \R$, where
$\Sy_d^-$ is the set of negative definite symmetric $d\times d$
matrices, and let 
\begin{equation}\label{paramquad}
 q(x,z):= \frac{1}{2} x\tp Q x + b\cdot x +c, \quad
\text{with}\;\; z=(Q,b,c)\in \Q_d\enspace, \end{equation}
be the quadratic form with parameter $z$ applied to the vector $x\in\R^d$.
Then for $g_T=q$, we have
\[ v^h(T,x)=\psi(x)=\sup_{z\in Z_T} g_T(x,z)\]
where $Z_T$ is a finite subset of $\Q_d$.
Then, in~\cite{mceneaney2011}, the following property is deduced from
a max-plus distributivity property, in the more general 
case where the Brownian motion does not have the same
dimension as the state space.

\begin{theorem}[\protect{\cite[Theorem 5.1]{mceneaney2011}}]
Assume that $\delta^m=0$, that $\sigma^m$ does not depend on $x$ and $u$,
that $f^m$ is affine with respect to $(x,u)$, that
$\ell^m$ is concave quadratic with respect to $(x,u)$, and that $\psi$ is 
the supremum of a finite number of concave quadratic forms.
Consider the time discretization of~\eqref{disc2} with~\eqref{defdiscop2}.
Then, for all $t\in\T_h$, there exists a set $Z_t$ 
and a  map $g_t:\R^d\times Z_t\to \R$ such that
for all $z\in Z_t$, $g_t(\cdot, z)$ is a concave quadratic form and
\[ v^h(t,x)=\sup_{z\in Z_t} g_t(x,z)\enspace .\]
Moreover, the sets $Z_t$ satisfy
\[ Z_t= \M\times 
\{\bar{z}_{t+h}:\W\to Z_{t+h}\mid \text{Borel measurable}\}\enspace ,\]
where $\W=\R^d$ is the space of values of the Brownian process.
\end{theorem}
Note that the sets $Z_t$ are infinite as soon as $t<T$.
However, if the Brownian process is discretized in space,
the set $\W$ can be replaced by a finite subset, 
and the sets $Z_t$ become finite.
Nevertheless, their cardinality increases exponentially as $t$ decreases:
$\card{Z_t}=\card{\M} \times( \card{Z_{t+h}})^p$ where $p$ is the cardinality of
the discretization of $\W$. Then, 
McEneaney, Kaise and Han proposed in~\cite{mceneaney2011} a pruning method
to reduce at each time step $t\in\T_h$ the cardinality of $Z_t$.

\section{Combining max-plus approximations and probabilistic schemes}

Here, we assume that the assumptions of the two previous sections
hold, and consider a time discretization scheme similar to the one
of Section~\ref{sec-touzi}.
The application of the operator  $T_{t,h}^m$ of~\eqref{def-th-m} to a function $\phi : \R^d \to \R, x\mapsto \phi(x) $ can be written, for each 
$x\in \R^d$, as
\begin{subequations}\label{def-T-G}
\begin{align}
T_{t,h}^m(\phi)(x)&=G_{t,x,h}^m(\tilde{\phi}^{m}_{t,x,h})\enspace , 
\end{align}
where $G_{t,x,h}^m$ is an operator from $\D$ to $\R$,
where $\D$ is the set of measurable functions from $\W:=\R^d$ to $\R$ 
with at most exponential growth rate, and 
\begin{align}
\tilde{\phi}^{m}_{t,x,h} &: \W \to \R, W \mapsto \phi(x + \underline{f}^m(x)h+
\underline{\sigma}^m(x) W) \enspace .
\label{def-tilde}
\end{align}
\end{subequations}
Indeed in the case of~\eqref{def-th-m}, $G_{t,x,h}^m$ is given by
\begin{eqnarray}
\lefteqn{G_{t,x,h}^m(\tilde{\phi})
= D_{m,t,x,h}^{0}(\tilde{\phi})}\nonumber\\
&&+ h \G^m(x, D_{m,t,x,h}^{0}(\tilde{\phi}),D_{m,t,x,h}^{1}(\tilde{\phi}),D_{m,t,x,h}^{2}(\tilde{\phi}))\label{defG}\end{eqnarray}
with, for $i=0,1,2$, 
\[ D_{m,t,x,h}^i(\tilde{\phi})=
\E(\tilde{\phi}(W_{t+h}-W_t) \Po^i_{m,t,x,h}(W_{t+h}-W_t))
\enspace .\]

Let us say that an operator $G:\D\to\R$ is monotone 
if it satisfies:
\begin{subequations}\label{monotone2}
\begin{align}
\phi, \psi \in \D,\; \phi \leq \psi \;\text{a.e.},\; \implies\; G(\phi)\leq G(\psi) \enspace,\end{align}
and that it is additively $\alpha$-subhomogeneous, for some
constant $\alpha>0$, if it satisfies:
\begin{align}
\forall K>0 \text{ and } \phi \in \D, \; G(\phi +K) \leq G(\phi) +\alpha K \enspace ,
\end{align}\end{subequations}
where $\phi+K$ is the map $W\in \W\mapsto \phi(W)+K$.
The above properties imply that the restriction of $G$ to
the set of bounded measurable functions is Lipschitz continuous with constant
$\alpha$.
Using the same kind of proof as in~\cite{touzi2011}
for~\eqref{monotone}, one can obtain the stronger property
that all the operators $G_{t,x,h}^m$ belong to the
class of monotone additively $\alpha_h$-subhomogeneous operators
from $\D$ to $\R$, for some constant $\alpha_h=1+Ch$ with $C\geq 0$.
This implies that $T_{t,h}^m$ sends the set of bounded measurable functions
to itself and is Lipschitz continuous with constant $\alpha_h$ on it.


Since we shall consider approximations of the value function
by suprema of concave quadratic forms, we shall need to apply the 
operators $T_{t,h}^m$ to quadratic forms, which are unbounded maps.
Therefore, in order 
to apply the above properties, we shall rather 
assume that the operators $T_{t,h}^m$ can also be written as 
\begin{equation}\label{def-wT}
 T_{t,h}^m(\phi)=\wT_{t,h}^m(\phi/\chi)\chi\end{equation}
where $\chi:\R^d\to \R$ is given by $\chi(x)=1+\|x\|^2$ 
with $\|\cdot\|$ beeing the Euclidian norm,
and where $\wT_{t,h}^m$ are also of the above form~\eqref{def-T-G}
with some operators $\wG_{t,x,h}^m$ instead of $G_{t,x,h}^m$ belonging again to the
class of monotone additively $\tilde{\alpha}_h$-subhomogeneous operators
from $\D$ to $\R$, for some constant $\tilde{\alpha}_h=1+\tilde{C}h$ 
with $\tilde{C}\geq 0$.
Under these conditions, We obtain 
a result similar to~\cite[Theorem 5.1]{mceneaney2011}.
\begin{theorem}\label{th-dist}
Consider the control problem of Section~\ref{sec-int}.
Assume that $\delta^m$ and $\sigma^m$ are constant, that $f^m$ 
is affine with respect to $(x,u)$, that
$\ell^m$ is concave quadratic with respect to $(x,u)$, and that $\psi$
is the supremum of a finite number of concave quadratic forms.
Consider the time discretization~\eqref{disc1}
with $\underline{\sigma}^m=\epsilon\sigma^m$,
$0< \epsilon\leq 1$ and $\underline{f}^m$ affine.
Assume that the operators $T_{t,h}^m$ satisfy~\eqref{def-wT}
for some operators $\wT_{t,h}^m$ of the form~\eqref{def-T-G},
with some operators $\wG_{t,x,h}^m$ instead of $G_{t,x,h}^m$ belonging to the
class of monotone additively $\tilde{\alpha}_h$-subhomogeneous operators
from $\D$ to $\R$, for some constant $\tilde{\alpha}_h=1+\tilde{C}h$ 
with $\tilde{C}\geq 0$, see~\eqref{monotone2}.
Assume also that the discretized value function $v_h$ of~\eqref{disc1}
is such that $v_h/\chi$ is bounded and Lipschitz continuous
with respect to $x$.
Then, for all $t\in\T_h$, there exists a set $Z_t$ 
and a  map $g_t:\R^d\times Z_t\to \R$ such that
for all $z\in Z_t$, $g_t(\cdot, z)$ is a concave quadratic form and
\[ v^h(t,x)=\sup_{z\in Z_t} g_t(x,z)\enspace .\]
Moreover, the sets $Z_t$ satisfy
\[ Z_t= \M\times 
\{\bar{z}_{t+h}:\W\to Z_{t+h}\mid \text{Borel measurable}\}\enspace .\]
\end{theorem}
This result uses the following properties, the
second one being a generalization of~\cite[Theorem 3.1]{mceneaney2011}.
\begin{lemma}\label{lem-imagequad}
Let $\tilde{z}$ be a measurable function from $\W$ to 
$\Q_d$. Let us consider the notations and assumptions
of Theorem~\ref{th-dist} and let $G_{t,x,h}^m$ be related to $T_{t,h}^m$ 
by~\eqref{defG}.
Let $\tilde{q}^{m, \tilde{z}}_{t,x,h}$ be the map
$\W\to\R,\; W \mapsto q(x + \underline{f}^m(x)h+
\underline{\sigma}^m(x) W, \tilde{z}(W) )$, with $q$ as in~\eqref{paramquad}.
Then, the function $x \mapsto G_{t,x,h}^m (\tilde{q}^{m, \tilde{z}}_{t,x,h})$ is a concave quadratic form, that is it can be written as $q(x,Z)$ for some $Z\in \Q_d$.
\end{lemma}
\begin{theorem}\label{main-theo}
Let $\W=\R^d$ and $G$ be a monotone additively $\alpha$-subhomogeneous operator
from $\D$ to $\R$, for some constant $\alpha>0$, see~\eqref{monotone2}.
Let $(Z,\A)$ be a mesurable space, and let $\W$ be endowed with its
Borel $\sigma$-algebra. 
Let $\phi:\W \times Z\to\R$ be mesurable map such that for all $z \in Z$, $\phi(\cdot,z)$ is bounded and continuous. 
Let $v:\W \to\R$ be such that
$v(W)=\sup_{z\in Z} \phi(W,z)$. Assume that v is continuous and 
bounded. 
Then,
\[ G(v)= \sup_{\bar{z}\in\overline{Z}}
G(\bar{\phi}^{\bar{z}})\] 
where $\bar{\phi}^{\bar{z}} :  \W\to \R,\; W \mapsto \phi(W,\bar{z}(W))$,
and 
\begin{align*}
\overline{Z}=&\{\bar{z}\, : \W \to Z,\; \text{measurable}\\
&\quad \text{and such that}\; \bar{\phi}^{\bar{z}} \; \text{is bounded} \}.
\end{align*}
\end{theorem}
\begin{proof}
Since $v$ is bounded and continuous, it belongs to $\D$, so that $G(v)$
is well defined. Similarly,
by definition, for all $\bar{z}\in \overline{Z}$, 
$\bar{\phi}^{\bar{z}} $ is measurable and bounded, so it belongs to $\D$, so that $G(\bar{\phi}^{\bar{z}})$ is well defined.

Let $\epsilon > 0$.
By definition of $v$, for all $W\in \W$, there exists $z^W \in Z$ such that $\phi(W,z^W) \ge v(W)-\epsilon$.
Then, since $W'\mapsto \phi(W',z_W)$ and $W'\mapsto v(W')$ are 
continuous maps $\W\to \R$, there exists $\delta^W>0$ such that
for all $W' \in B(W, \delta^W)$ (the open ball centered at $W$ with radius
$\delta^W$), $|\phi(W',z^W)-\phi(W,z^W)| \le \epsilon$ and $|v(W')-v(W)| \le \epsilon$. Then,  for $W'\in B(W, \delta^W)$, we have 
\[\phi(W',z^W) \ge \phi(W,z^W)-\epsilon \ge v(W)-2\epsilon \ge v(W')-3\epsilon\enspace .\] 
As $\W$ is the countable union of compact metric spaces,
there exists a sequence $(W_i)_{i\geq 0}$ of $\W$ such that
$\W= \cup_{i\geq 0} B(W_i, \delta^{W_i})$.
Let us denote, for all $i\geq 0$, $\W_i=B(W_i, \delta^{W_i})$ and $\W'_i=\W_i\setminus(\cup_{j<i} \W_j)$. 
Define the function $\bar{z}$ such that,
for all $i\geq 0$, $\bar{z}(W')=z^{W_i}$, for $W' \in \W'_i$.
Since $(\W'_i)_{i\geq 0}$ is a countable partition of $\W$ 
composed of Borel sets, the map $\bar{z}$ is well defined on $\W$ and
measurable.
Moreover, by the above properties and the definition of $v$, we have 
\[ v(W)\geq \bar{\phi}^{\bar{z}}(W) =\phi(W,\bar{z}(W)) \ge v(W)-3\epsilon,
\quad \forall W \in \W\enspace .\]
Since $v$ is bounded, this implies that $\bar{\phi}^{\bar{z}}$ is bounded,
which implies that $\bar{z}$ belongs to $\overline{Z}$.

Since $G$ is monotone and additively $\alpha$-subhomogeneous from $\D$ to $\R$,
and $\epsilon>0$, we get that
\[ G(v)\geq G(\bar{\phi}^{\bar{z}}) \ge G(v-3\epsilon)\ge
G(v)-3\alpha\epsilon \enspace .\]
Then
\[ G(v)\geq \sup_{\bar{z}\in\overline{Z}} G(\bar{\phi}^{\bar{z}}) \ge
G(v)-3\alpha\epsilon \enspace ,\]
and since this property holds for all $\epsilon>0$, we 
obtain the equality, which shows  the assertion of the theorem.
\end{proof}

Using Theorem~\ref{th-dist}, we get that $v^h(t,\cdot)$ is 
the supremum of concave quadratic maps,
but as in~\cite[Theorem 5.1]{mceneaney2011},
the sets $Z_t$ are infinite for $t<T$.
Here, we shall compute the expression of 
the maps $v^h(t,\cdot)$  by approximating the operators $T_{t,h}^m$ 
as in the same spirit as in~\cite{touzi2011}, that is using 
the simulation of the processes $\hat{X}^{m}$.
The main difference with the method of~\cite{touzi2011}
is that the egression estimations are done on quadratic forms and
not on the value functions directly.
Because of the simulations, 
we should only need to compute the values $v^h(t,\hat{X}^{m}(t))$.
This means that if $N$ is the number of samples of the
Brownian process, then the number of quadratic forms 
$g_t(\cdot,z)$ that are essentials in the computation
of $v^h(t,\hat{X}^{m}(t))$ as a supremum of quadratic forms
 is less or equal to $N\times M$, where $M$ is the cardinality of $\M$.
Then, for any random quadratic form 
which is optimal for a particular $\hat{X}^{m}(t)$,
we need to compute its image by $T_{t,h}^m$.
In~\cite{mceneaney2011}, this image is obtained by hand. Here, we shall
rather use a regression estimation.
The result of Lemma~\ref{lem-imagequad} implies that taking 
for the linear regression space, the space of quadratic forms 
gives an exact result at least when the sample size is large
(to ensure that the solution of the estimation problem is unique).

With all these properties in mind, we construct the following algorithm.
Let us denote by $S_{t,h}^m(x,w)$ the following operator
which sends  $(\hat{X}^{m}(t),W_{t+h} -W_t)$ into
$\hat{X}^{m}(t+h)$:
\[S_{t,h}^m(x,w)= x+ \underline{f}^m(x)h+\underline{\sigma}^m(x) w\enspace .\]
Since $\underline{\sigma}^m$ is constant and
$\underline{f}^m$ is affine, the map $S_{t,h}^m(\cdot,w)$ is affine for all 
$w\in W$.

\noindent
\hrulefill \
\begin{algorithm}\ 

\noindent
\emph{Input:} A constant $\eps$ giving the precision,
and a $5$-uple $N=(\Nzero,\NXW,\NX,\NW,\Nm)$ of integers giving the numbers of samples
and the ``method of sampling'' $\Nm\in\{1,\ldots, 5\}$ described below.
A finite subset $Z_T$ of $\Q_d$ such that
$|\psi(x)-\max_{z\in Z_T} q(x,z)|\leq \eps$, for all $x\in\R^d$,
and $\card{Z_T}\leq M\times \Nzero$,
 and the operators $T_{t,h}^m$ and $G_{t,x,h}^m$  as in Theorem~\ref{th-dist}.

\noindent
\emph{Output:}
The subsets $Z_t$ of $\Q_d$, for $t\in\T_h\cup\{T\}$, 
and the approximate value function
$v^{h,N}:(\T_h\cup\{T\})\times \R^d\to \R$.

\noindent
$\bullet$ \emph{Initialization:}
Let  $\hat{X}^{m}(0)$ be random and independent of the Brownian process.
Consider a sample of $(\hat{X}^{m}(0),(W_{t+h}-W_t)_{t\in \T_h})$
of size $\Nzero$ indexed by $\omega\in \Omega_{\Nzero}:=\{1,\ldots, {\Nzero}\}$,
and denote, for each $t\in\T_h\cup\{T\}$ and
$\omega \in \Omega_{\Nzero}$, $\hat{X}^{m}(t,\omega)$ the value
of  $\hat{X}^{m}(t)$ induced by this sample and satisfying~\eqref{xhat}.
Define $v^{h,N}(T,x)=\max_{z\in Z_T} q(x,z)$,
for $x\in \R^d$, with $q$ as in~\eqref{paramquad}.

\noindent
$\bullet$ For $t=T-h,T-2h,\ldots, 0$ apply the following 3 steps:

(1) For each $\omega\in\Omega_{\Nzero}$ and $m\in \M$,
construct a sample $(\omega_1,\omega'_1),\ldots, (\omega_{\NXW},\omega'_{\NXW})$
of elements of $\Omega_{\Nzero}\times \Omega_{\Nzero}$,
using the method $\Nm$ and possibly
the constants $\NX$ and $\NW$.
Induce the sample $\hat{X}^{m}(t,\omega_i)$ (resp.\ $(W_{t+h}-W_t)(\omega'_i)$)
for $i\in \Omega_{\NXW}$
of $\hat{X}^{m}(t)$ (resp.\ $W_{t+h}-W_t$).
Denote by $\W^N_t\subset \W$
the set of $(W_{t+h}-W_t)(\omega'_i)$ for $i \in \Omega_{\NXW}$.


(2) For each $\omega\in\Omega_{\Nzero}$ and $m\in \M$, 
construct $z_t\in \Q_d$ depending on  $\omega$ and  $m$ as follows:

Let
$\bar{z}_{t+h}:\W^N_t\to Z_{t+h}\subset \Q_d$ be such that,
for all $i\in \Omega_{\NXW}$ we have
\begin{align*}
& v^{h,N}(t+h,S^m_{t,h}(\hat{X}^{m}(t,\omega),(W_{t+h} -W_t)(\omega'_i)))\\
&\; = q\big(S^m_{t,h}(\hat{X}^{m}(t,\omega),(W_{t+h} -W_t)(\omega'_i))
,\bar{z}_{t+h}((W_{t+h} -W_t)(\omega'_i))\big)\enspace .\end{align*}
Extend $\bar{z}_{t+h}$ as a measurable map on $\W$.
Let $\tilde{q}^{m, \bar{z}}_{t,x,h}$ be as in Lemma~\ref{lem-imagequad},
that is be the map $\W\to\R,\; W \mapsto q(x + \underline{f}^m(x)h+
\underline{\sigma}^m(x) W, \bar{z}(W) )$.
Compute an approximation of 
$x \mapsto G_{t,x,h}^m (\tilde{q}^{m, \bar{z}}_{t,x,h})$ by a regression
estimation on the set of quadratic forms using the sample
$(\hat{X}^{m}(t,\omega_i), (W_{t+h} -W_t)(\omega'_i))$, with $i\in\Omega_{\NXW}$.
We obtain $z_t\in \Q_d$ such that
$q(x,z_t)\simeq G_{t,x,h}^m (\tilde{q}^{m, \bar{z}}_{t,x,h})$.

(3) Denote by $Z_t$ the set of all the $z_t\in \Q_d$ obtained in this way,
and define
\[ v^{h,N}(t,x)=\max_{z\in Z_t} q(x,z)\quad
\forall x\in \R^d \enspace .\]

\end{algorithm}
\hrulefill \

Let us precise now the different choices of the ``method of sampling'' 
$\Nm$  used in the algorithm:
\catcode`\@=11
\def\refcounter#1{\protected@edef\@currentlabel
       {\csname the#1\endcsname}%
}
\catcode`\@=12
\newcounter{method}
\def\themethod{{\rm Method \arabic{method}:}}
\def\myitem{\refstepcounter{method}\item[\themethod]}

\begin{itemize}[\setlabelwidth{Method 1:}]
\myitem Assume $\NXW=\Nzero$ and take $\omega_i=\omega'_i=i$ for
$i \in \Omega_{\NXW}$, which means that we take the initial sampling.
\myitem Assume $\NXW=\NX\times \NW$, and choose
once for all $\omega\in \Omega_{\Nzero}$ and $m\in \M$ in the algorithm:
a random sampling $\omega_{i,1},\; i=1,\ldots, \NX$
among the elements of $\Omega_{\Nzero}$
and independently a random sampling  $\omega'_{1,j}\; j=1,\ldots, \NW$
among the elements of $\Omega_{\Nzero}$,
then take the product of samplings, leading to $(\omega_{i,1},\omega'_{1,j})$
for $i=1,\ldots, \NW$ and $j=1,\ldots, \NW$. Reindexing the sampling,
we obtain $(\omega_i,\omega'_i)$ for $i=1,\ldots, \NXW$.
\myitem Do as in Method 2, but choose different samplings
for each   $\omega\in \Omega_{\Nzero}$ and $m\in \M$ in the algorithm,
independently.
\myitem Assume $\NXW=\NX\times \NW$ and $\NW=\Nzero$ and do as in Method 2, but 
take the fixed sampling $\omega'_{1,j}=j$ instead of random sampling.
\myitem Assume $\NXW=\Nzero^2$ and do as in Method 2, but 
take the fixed samplings $\omega_{i,1}=i$ and
$\omega'_{1,j}=j$ instead of random samplings.
\end{itemize}

It is easy to see that the sets 
$Z_t$ of the above algorithm  satisfy
$\card{Z_t}\leq M \times \Nzero$ for all $t\in\T_h$.
Then, the number of computations at each time step for the optimization 
(computation of the $\bar{z}_{t+h}$)
 will be at most in the order of
$(M\times \Nzero)^2\times \NXW$  and at most in the order of
$(M\times \Nzero)^2\times \NW$ when using methods 2,3,4.
Moreover, the number of computations at each time step for the
regression estimation will be at most in the order of 
$\M\times \Nzero \times \NXW$ so will be negligeable with respect to
the optimization step.

Note that in the above algorithm, the regression estimation depends 
 only on the value of $\bar{z}_{t+h}$ on the simulations 
$(W_{t+h} -W_t)(\omega_i))$, with $i\in\Omega_{\NXW}$. That is 
the extension of $\bar{z}_{t+h}$ to a measurable function on $\W$
was only needed for the definition of $\tilde{q}^{m, \bar{z}}_{t,x,h}$
and $x \mapsto G_{t,x,h}^m (\tilde{q}^{m, \bar{z}}_{t,x,h})$.
From Lemma~\ref{lem-imagequad},
$p : x \mapsto G_{t,x,h}^m (\tilde{q}^{m, \bar{z}}_{t,x,h})$ is a quadratic form. 
Hence, when $\NXW$ is large, the regression estimation of $p$
as a quadratic form using the given sample is a good approximation.
Under these conditions, we have
$v^{h,N}(t,x)= \sup_{m\in \M} v^{h,N,m}(t,x)$ for all $x$,
with $v^{h,N,m}(t,x)\simeq T_{t,h}^m( v^{h,N}(t+h,\cdot))(x)$ holding
for all $x=\hat{X}^{m}(t,\omega)$ with $\omega\in\Omega_{\Nzero}$. 
This implies that, when the regression approximations converge, 
$v^{h,N}$ is a good approximation of $v^h$.
Then, under the assumptions of Theorem~\ref{th-dist},
one may expect a convergence result comparable to
the one of~\cite{touzi2011}, 
showing the existence of some 5-uples $N_h$ such that
$v^{h,N_h}$ converges towards the value 
function of the control problem when $h$ goes to $0$.
We present numerical tests to confirm this convergence in the next section,
at least in the case of some of the sampling methods proposed above.
The precise convergence study is left for further work.

\section{Numerical tests}

\addtolength{\textheight}{-1.5cm}   

To test our algorithm, 
we consider the problem of pricing and hedging an option 
with uncertain volatility and two underlying processes, 
studied as an example in Section 3.2 of~\cite{kharroubi2014}.
There, the method proposed is based on a regression on 
a process involving not only the state but also the (discrete) control.

With the notations of the introduction, we consider the case where
$d=2$, $\M = \{\rho_{\min}, \rho_{\max}\}$ with $ -1 \le \rho_{min}, \rho_{max} \le 1$ ,   and there is no continuum control, so $u$ is ommited.
The dynamics of the processes are given,
for all  $m \in \M$, by 
$f^{m}=0$, and for $\xi=(\xi_1,\xi_2)\in \R^2$, 
\[ \sigma^{m}(\xi) = \left[ \begin{array}{cc}
\sigma_1 \xi_1 & 0 \\
\sigma_2 m \xi_2 & \sigma_2 \sqrt{1-m^2} \xi_2 \\ \end{array}\right] \]
with $\sigma_1, \sigma_2 > 0$.
The parameters of the reward satisfy $\delta^m=0$, $l^m=0$, 
and, for $\xi=(\xi_1,\xi_2)\in \R^2$, 
\[\psi(\xi)=(\xi_1-\xi_2-K_1)^+-(\xi_1-\xi_2-K_2)^+\]
with $x^+ = \max(x,0)$, $K_1 < K_2$.

The two coordinates of the controlled process stay in $\R_+$,
the set of positive reals.
To be in the conditions of Theorem~\ref{th-dist}, we approximate the function $\psi$ with a supremum of a finite number of concave quadratic forms on a large subset of $\R_+^2$, typically on the set of $\xi$ such that $\xi_1-\xi_2\in [-100,100]$.
Note that since the second derivative of $\psi$ is $-\infty$ in some points,
it is not  $c$-semiconvex for any $c>0$ and bounded domain, so
the approximation need to use some quadratic forms with a large negative 
curvature, and so the algorithm proposed in~\cite{mceneaney2011} may not work.
The maps $\sigma^m$ for $m \in  \M$ are not constant but they are linear, so one can show that the result of Theorem~\ref{th-dist} still holds.

We take the same constants as in~\cite{kharroubi2014}:
$\sigma_1=0.4,\; \sigma_2=0.3, \; K_1=-5,\; K_2=5,\; T=0.25$,
$\rho_{\min}=-0.8, \rho_{\max}=0.8$.
We fix the time discretization step to $h=0.1$.

We first tested our algorithm in the case where $\M$ is the singleton
$\{\rho_{\min}\}$ or $\{\rho_{\max}\}$, which means that there is no action
on the process, so that the true value function can be computed 
analytically, and compared with the solution obtained by our algorithm.
The method $\Nm=1$ gives very bad results even at time $T-h$.
The method $\Nm=5$ need too much space and time even for $\Nzero=1000$.
In Table~\ref{table1}, we present  for different
values of $N=(\Nzero,\NXW,\NX,\NW,\Nm)$, with $\Nm=2,3,4$,
the norm of the error
on the value function at time $t=0$ and states $\xi_2=50$ 
and $\xi_1\in [20,80]$. 
We see that the best method is the second one,
and that Method 3 gives very bad results.
This may be explained by the introduction of 
a biais due to the maximization of independent random variables.
Note also that the errors for Method 2 are comparable
to the standard deviations obtained
in~\cite{gobet} by Gobet, Lemor and Warin in the case of 
similar option problems with a usual regression estimation of
the value function.

\begin{table}[htb]\label{table1}
\begin{tabular}{|c|c|c|c|c|c|c|c|}\hline
$\rho$&$\Nzero$ &$\NXW$ &$\NX$ &$\NW$ &$\Nm$ & $e_\infty$ & $e_1$ \\ \hline
-0.8&1000&10000&10&1000&2& 0.521 &   0.173   \\ \hline
0.8&1000&10000&10&1000&2& 0.157 & 0.074 \\ \hline
-0.8&1000&1000&10&100&2&0.75 & 0.41  \\ \hline
0.8&1000&1000&10&100&2&0.36 & 0.11\\ \hline
-0.8&1000&1000&10&100&3&3.48 & 1.92  \\ \hline
0.8&1000&1000&10&100&3&3.05 & 0.81\\ \hline
-0.8&100&1000&10&100&2& 1.95 &0.46 \\ \hline
0.8&100&1000&10&100&2& 1.81& 0.33 \\ \hline
-0.8&100&10000&10&1000&2& 2.09 & 0.53 \\ \hline
0.8&100&10000&10&1000&2& 1.79 & 0.36 \\ \hline
-0.8&100&1000&10&100&4& 2.15& 0.55\\ \hline
0.8&100&1000&10&100&4& 1.80&  0.39\\ \hline
\end{tabular}
\caption{Sup-norm and normalized $\ell^1$ norm of the error,
on the value function with constant $\rho$,
at time $t=0$, and states $\xi_2=50$ and $\xi_1\in [20,80]$,
denoted $e_\infty$ and  $e_1$ resp.}
\end{table}

In view of these results, we present in Figure~\ref{fig1000} the result obtained for
the control problem tested in~\cite{kharroubi2014}, that is
with $\M = \{\rho_{\min}, \rho_{\max}\}$,
and $\Nzero=1000$, $\NXW=\NX\times \NW$, $\NX=10$, $\NW=1000$
and $\Nm=2$.
The result is very similar to the one presented in~\cite{kharroubi2014}.

\figperso{\includegraphics[scale=0.35]{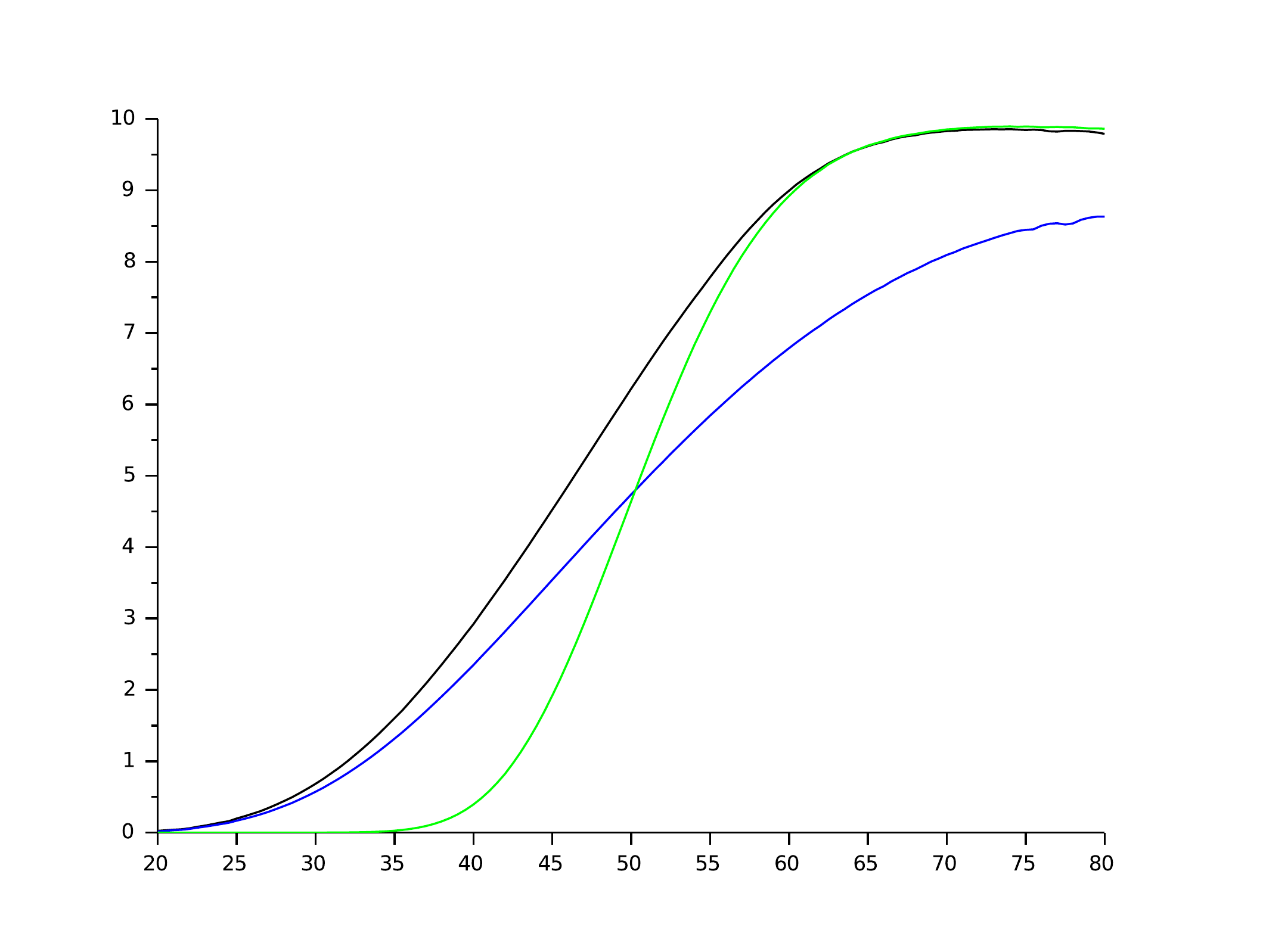} 
}{\caption{
Value function obtained at $t=0$, and $\xi_2=50$ as a function 
of $\xi_1\in [20,80]$.
Here $\Nzero=1000$, $\NXW=\NX\times \NW$, $\NX=10$, $\NW=1000$
and $\Nm=2$.
In blue, $\rho$ is constant equal to $-0.8$, in green $\rho$ is constant
equal to $0.8$, and in black $\rho\in \{-0.8,0.8\}$.\label{fig1000}}}

\begin{figure}

\end{figure}

\addtolength{\textheight}{-3cm}   


\bibliography{maxproba}

\begin{thebibliography}{1}
\providecommand{\url}[1]{#1}
\csname url@rmstyle\endcsname
\providecommand{\newblock}{\relax}
\providecommand{\bibinfo}[2]{#2}
\providecommand\BIBentrySTDinterwordspacing{\spaceskip=0pt\relax}
\providecommand\BIBentryALTinterwordstretchfactor{4}
\providecommand\BIBentryALTinterwordspacing{\spaceskip=\fontdimen2\font plus
\BIBentryALTinterwordstretchfactor\fontdimen3\font minus
  \fontdimen4\font\relax}
\providecommand\BIBforeignlanguage[2]{{%
\expandafter\ifx\csname l@#1\endcsname\relax
\typeout{** WARNING: IEEEtran.bst: No hyphenation pattern has been}%
\typeout{** loaded for the language `#1'. Using the pattern for}%
\typeout{** the default language instead.}%
\else
\language=\csname l@#1\endcsname
\fi
#2}}

\bibitem{touzi2011}
\BIBentryALTinterwordspacing
A.~Fahim, N.~Touzi, and X.~Warin, ``A probabilistic numerical method for fully
  nonlinear parabolic {PDE}s,'' \emph{Ann. Appl. Probab.}, vol.~21, no.~4, pp.
  1322--1364, 2011. [Online]. Available:
  \url{http://dx.doi.org/10.1214/10-AAP723}
\BIBentrySTDinterwordspacing

\bibitem{cheridito2007}
\BIBentryALTinterwordspacing
P.~Cheridito, H.~M. Soner, N.~Touzi, and N.~Victoir, ``Second-order backward
  stochastic differential equations and fully nonlinear parabolic {PDE}s,''
  \emph{Comm. Pure Appl. Math.}, vol.~60, no.~7, pp. 1081--1110, 2007.
  [Online]. Available: \url{http://dx.doi.org/10.1002/cpa.20168}
\BIBentrySTDinterwordspacing

\bibitem{mceneaney2010}
H.~Kaise and W.~M. McEneaney, ``Idempotent expansions for continuous-time
  stochastic control: compact control space,'' in \emph{Proceedings of the 49th
  IEEE Conference on Decision and Control}, Atlanta, Dec. 2010.

\bibitem{mceneaney2011}
W.~M. McEneaney, H.~Kaise, and S.~H. Han, ``Idempotent method for
  continuous-time stochastic control and complexity attenuation,'' in
  \emph{Proceedings of the 18th IFAC World Congress, 2011}, Milano, Italie,
  2011, pp. 3216--3221.

\bibitem{kharroubi2014}
\BIBentryALTinterwordspacing
I.~Kharroubi, N.~Langren{\'e}, and H.~Pham, ``A numerical algorithm for fully
  nonlinear {HJB} equations: an approach by control randomization,''
  \emph{Monte Carlo Methods Appl.}, vol.~20, no.~2, pp. 145--165, 2014.
  [Online]. Available: \url{http://dx.doi.org/10.1515/mcma-2013-0024}
\BIBentrySTDinterwordspacing

\bibitem{bouch-tou}
\BIBentryALTinterwordspacing
B.~Bouchard and N.~Touzi, ``Discrete-time approximation and {M}onte-{C}arlo
  simulation of backward stochastic differential equations,'' \emph{Stochastic
  Process. Appl.}, vol. 111, no.~2, pp. 175--206, 2004. [Online]. Available:
  \url{http://dx.doi.org/10.1016/j.spa.2004.01.001}
\BIBentrySTDinterwordspacing

\bibitem{gobet}
E.~Gobet, J.-P. Lemor, and X.~Warin, ``A regression-based monte carlo method to
  solve backward stochastic differential equations,'' \emph{The Annals of
  Applied Probability}, vol.~15, no.~3, p. 2172–2202, 2005.

\end{thebibliography}
\bibliographystyle{IEEEtran}

\end{document}